\documentclass{article}
\usepackage{graphicx} 

\usepackage[utf8]{inputenc}
\usepackage{amsmath}
\usepackage{amsthm}
\usepackage{amsfonts}
\usepackage{amssymb}
\usepackage{tikz}
\usepackage{graphicx}
\usepackage{pgfplots}
\usepackage{xcolor}
\usepackage{enumitem}
\usepackage{mathtools}
\pgfplotsset{compat=1.11}
\graphicspath{ {./images/} }
\usepackage{float}
\tikzset{every picture/.style={line width=0.75pt}} 
\makeatletter
\usepackage{biblatex} 
\usepackage{hyperref}
\hypersetup{
    colorlinks=true,
    linkcolor=blue,
    filecolor=magenta,      
    urlcolor=cyan,
    pdftitle={Overleaf Example},
    pdfpagemode=FullScreen,
    }

\newcommand{\I}{\operatorname{\mathscr{I}}}
\newcommand{\Burch}{\operatorname{Burch}}
\newcommand{\BI}{\operatorname{BI}}
\newcommand{\RS}{\mathfrak{R}}
\newcommand{\rs}{\mathfrak{r}}

\def\legendre@dash#1#2{\hb@xt@#1{%
  \kern-#2\p@
  \cleaders\hbox{\kern.5\p@
    \vrule\@height.2\p@\@depth.2\p@\@width\p@
    \kern.5\p@}\hfil
  \kern-#2\p@
  }}
\def\@legendre#1#2#3#4#5{\mathopen{}\left(
  \sbox\z@{$\genfrac{}{}{0pt}{#1}{#3#4}{#3#5}$}%
  \dimen@=\wd\z@
  \kern-\p@\vcenter{\box0}\kern-\dimen@\vcenter{\legendre@dash\dimen@{#2}}\kern-\p@
  \right)\mathclose{}}
\newcommand\legendre[2]{\mathchoice
  {\@legendre{0}{1}{}{#1}{#2}}
  {\@legendre{1}{.5}{\vphantom{1}}{#1}{#2}}
  {\@legendre{2}{0}{\vphantom{1}}{#1}{#2}}
  {\@legendre{3}{0}{\vphantom{1}}{#1}{#2}}
}
\def\dlegendre{\@legendre{0}{1}{}}
\def\tlegendre{\@legendre{1}{0.5}{\vphantom{1}}}
\makeatother

\newtheorem{theorem}{Theorem}
\newtheorem*{theorem*}{Theorem}
\theoremstyle{definition}
\newtheorem{defn}{Definition}
\newtheorem{lemma}{Lemma}

\newtheorem{remark}{Remark}
\newtheorem{example}{Example}
\newtheorem{prop}{Proposition}
\newtheorem{corollary}{Corollary}

\usepackage{mathrsfs}

\title{Generalizations of Burch Ideals and Ideal-Periodicity}
\author{Tejas Rao}

\begin{document}

\maketitle 
Consider a minimal free resolution of a module $M$ over a local Noetherian ring $R$. Over such rings, resolutions are often infinite, for example by the The Auslander-Buchsbaum formula when $\operatorname{depth}(R)=0$ \cite{aus}. The question of periodicity in infinite resolutions is the subject of intensive research for example in the works of Eisenbud, Peeva, and Gasharov, and the central survey of Avramov \cite{avr,eishom,pee}. 

The weaker question of whether the ideals of minors of maps in these resolutions are periodic is more recent. Dao, Kobayashi, and Takahashi, introduced an invariant of depth $0$ rings called the Burch Index, among other things proving that certain conditions allowed for direct summands to be present in a step in a resolution \cite{dao2020burch}. Applying and these techniques, Eisenbud and Dao showed that the $1\times 1$ minors of modules over a depth $0$ local ring $R$ of embedding dimension $\geq 2$ are periodic provided that the Burch index is at least $2$ \cite{eisbur}. More specifically, they showed that in this case the syzygies $\operatorname{syz}_n^R(M)$ in the resolution have $k$ as a direct summand for all sufficiently large $n$, and simultaneously that the ideals of $1\times 1$ minors are asymptotically all $\mathfrak{m}$ \cite[Theorem 4.1]{eisbur}. Brown, Dao, and Sridhar further researched this \emph{ideal-periodicity}, proving $2$-periodicity over complete intersections and Golod rings \cite{dao2}. 

The case of periodicity for Burch index $0$ and $1$ local depth $0$ rings is still open, and will be a major part of this paper. We will introduce certain generalizations of Burch Indices, which allow one to prove periodicity in classes of Burch Index $1$ and $0$ rings. 
In addition, these generalizations often make sense in positive depth Noetherian local rings, and periodicity is proven in some such rings as well. Extensive calculations are utilized, entirely in Macaulay2 \cite{M2}. 

Throughout this paper we consider regular local rings $(S,\mathfrak{n},k)$ with $S=k[[x_1,...,x_n]]$, and corresponding reductions $(R,\mathfrak{m},k)$ with $R=S/I$. We will write minimal $R$-resolutions $\epsilon: (F_\bullet,A_\bullet)\rightarrow M$ and consider the $R$-ideals generated by $1\times 1$ minors of the matrices: $\I_1(A_1)$. The primary object of study is the $N$-Burch Ideal
$$\BI_N(I):=\mathfrak{n}I:(I:N)$$
and corresponding $N$-Burch Index
$$\Burch_N(I):=\operatorname{length}_S(N/(\BI_N(I)\cap N))$$
We will consider iterated Burch Indices
$$\Burch^j(I):=\Burch_{\BI^{j-1}(I)}(I)$$
where $\BI^{0}(I):=\mathfrak{m}$. Let $n$ be the first index such that $\Burch^n(I)=0$ and consider the generalized Burch Index, given as
$$\operatorname{gb}(I)=\begin{cases}
    \max \{\Burch^j(I)\}_{j<n} & \text{if $n\neq 1$,}\\
    0 & \text{otherwise.}
\end{cases}$$
If there is no such $n$, we take the supremum of all $\Burch^j(I)$. Also let Burch depth be $$\operatorname{bd}(I)=\sup \{j | \Burch^i(I)=1 \text{ for } i\leq j\}$$ With this terminology, the primary theorem of Eisenbud and Dao's paper on Burch Rings, Theorem $4.1$, states that any resolution $\epsilon: (F_\bullet,A_\bullet)\rightarrow M$ over a ring $R=S/I$ for which $\operatorname{gb}(I)\geq 2$ and $\operatorname{bd}(I)=0$, satisfies $\I_1(A_m)=\mathfrak{m}$ for $m>>0$ \cite{eisbur}. The first main result of this paper extends this to arbitrary Burch Depth: 
\begin{theorem}\label{Big1}
    If $\operatorname{gb}(I)\geq 2$, then all minimal resolutions $\epsilon: (F_\bullet, A_\bullet)\rightarrow M$ of modules over $R$ satisfy $\I_1(A_m)+\I_1(A_{m+1})=N$ for some fixed ideal $N$ and $m>>0$. In particular, $N=\BI^j(I)$ for some $0\leq j\leq \operatorname{bd}(I)$. 
\end{theorem}

To prove this we will utilize two key propositions, Lemma \ref{dual} and Lemma \ref{Up}. This theorem will be proven in Section \ref{kemma}. Parsing what $\operatorname{gb}(I)\geq 2$ implies, we must find ideals $I\subset S$ such that $\Burch^j(I)=1$ for $1\leq j<n$ and $\Burch^n(I)\geq 2$. Here is such an example: 

\begin{example}\label{one}
    Let $S=k[[x_1,...,x_{m-1},y]]$ and $I=(x_1y,x_2y,...,x_{m-1}y,y^{n+1})$. Notice that the rings $S/I$ are not Cohen-Macaulay. Then $\Burch^j(I)=1$ for $1\leq j<n$, and $\Burch^n(I)=m$ (one can in fact take any $m\geq 2$, and when $m=1$ the construction works as well but of course with $\Burch^n(I)=1$). In particular, one can compute 
\begin{align*}
        \BI^1(I)&=(x_1,x_2,...,x_{m-1},y^2)\\
        \BI^2(I)&=(x_1,x_2,...,x_{m-1},y^3)\\
        &...\\
        \BI^{n-1}(I)&=(x_1,x_2,...,x_{m-1},y^{n})\\
        \BI^{n}(I)&=I+(x_1,x_2,...,x_{m-1})^2
    \end{align*}
    because
\begin{align*}
    (I:n)&=(x_1y,x_2y,...,x_{m-1}y,y^{n})\\
    (I:\BI^1(I))&=(x_1y,x_2y,...,x_{m-1}y,y^{n-1})\\
    &...\\
    (I:\BI^{n-2}(I))&=(x_1y,x_2y,...,x_{m-1}y,y^{2})\\
    (I:\BI^{n-1}(I))&=(y)
\end{align*}
   

One may also note that $\BI^j(I)=\BI^n(I)$ for $j\geq n$ since in this range $(I:\BI^j(I))=(y)$. Thus, $\Burch^j(I)=0$ for $j>n$. In particular, $\operatorname{gb}(I)=m$ and $\operatorname{bd}(I)=n$.

\end{example}

Our second main result is related to a notion of \textit{untwisting}. In particular, we develop under certain conditions a column-wise Burch approach, where we need only positive of $\Burch_{\I_1(c)}(I)$ for $\I_1(c)$ the ideal generated by entries of some column in a matrix in a resolution, in Lemma \ref{dual}. Intuitively, the ideals $\I_1(c)$ must be 'small' for this approach to be powerful. Thus when $\I_1(c)$ is large for each column in a matrix, we develop Lemma \ref{twist1} to, under certain conditions, break apart these columns into smaller ideals $N$, and determine periodicity by considering $\BI_{N}(I)$. This culminates in the second main result of this paper:

\begin{theorem}\label{Big2}
    Fix $S=k[[x_1,...,x_n]]$ and an ideal $I$. Assume for each $i$, there exists some $j\neq i$ and there exists some $\alpha$ such that $\alpha x_j$ and $\alpha x_i$ are minimal generators of $I$. Then for any minimal $R$-resolution $\epsilon: (F_\bullet,A_\bullet)\rightarrow M$, if for each $x_j\in \{x_1,...,x_n\}$ there exists an index $m$ such that $(0)\subsetneq \I_1(c_m)\subset (x_j)$, $$\I_1(A_a)=\mathfrak{m}$$for $a>>0$.  
\end{theorem}

This theorem applies to some Burch Index $0$ rings, as well as positive depth local rings: 

\begin{example}
    Let $S=k[[x,y,z,w]]$ and $I=(xz,yz,zw,xw)$. The conditions of the above theorem are satisfied, but as $\mathfrak{m}\subset R$ has zero annihilator, $\operatorname{depth}(R)>0$. In particular, $\BI(I)=0$. We have (after modding out by $I$), 
    \begin{align*}
        \BI_{(x)}(I)&=(xy,x^2)\\
        \BI_{(y)}(I)&=(w^2,yw,y^2,xy,x^2)\\
        \BI_{(z)}(I)&=(z^2)\\
        \BI_{(w)}(I)&=(w^2,yw)
    \end{align*}
    We choose columns contained in each Burch ideal to test the above theorem. In particular, consider the resolution $\epsilon: (F_\bullet,A_\bullet)\rightarrow R/J$ where $J=(x^2y^2,z^3,yw)$. Note $x^2y^2$ is in the first two ideals, and $z^3$ is in the third one and $yw$ is in the fourth. With Macaulay2 \cite{M2}, we find $\I_1(A_2)=(x,y,z,w)$, supporting the theorem. Lemma \ref{dual} (the column-wise Burch lemma) ensures this ideal persists asymptotically.
\end{example}

\section{$N$-Burch Ideals}

In this section we will flesh out some of the details of the introduction, and prove some initial results. Throughout this paper, we let $(S,\mathfrak{n},k)$ be a regular local ring, and for an ideal $I\subset S$ write $R=S/I$ as a local ring $(R,\mathfrak{m},k)$. The main thrust of the paper of Eisenbud and Dao is to consider ideals of the form 
\begin{align*}
    \BI(I)&=I\mathfrak{n}:(I:\mathfrak{n})
\end{align*}
called the \emph{Burch Ideal} \cite{eisbur}. Eisenbud and Dao restrict to the case where $\operatorname{depth}(S/I)=0$ and $I\neq 0$ so that 
$$\mathfrak{n}^2\subset \BI(I)\subset \mathfrak{n}$$
This allows us to define the Burch index as $$\Burch(I)=\dim_k(\mathfrak{n}/BI_S(I))$$
If we instead start with an arbitrary depth $0$ local ring $R$, we can write $\hat{R}=S/I$ as some minimal regular presentation of the $\mathfrak{m}$-adic completion $\hat{R}$, and compute $\Burch(R)=\Burch(I)$. Similarly, we write $\BI(R)=\BI(I)\hat{R}\cap R$. Theorem $2.3$ of Eisenbud and Dao states that $\Burch(R), \BI(R)$ are well-defined, independent of choice of presentation \cite[Theorem 2.3]{eisbur}. 

The main result of \cite{eisbur} is the following 
\begin{theorem}[Eisenbud and Dao (Theorem 4.1)]
    Let $(R,\mathfrak{m},k)$ be a local ring of depth $0$ and embedding dimension $\geq 2$. For every non-free $R$-module $M$:
    \begin{align*}
        &\text{(1) If $\Burch(R)\geq 2$ then $k$ is a direct summand of $\operatorname{syz}_i^R(M)$}\\ &\text{for some $i\leq 5$ and for all $i\geq 7$}\\
        &\text{(2) If $\Burch(R)\geq 1$ and $k$ is a direct summand of $\operatorname{syz}_s^R(BI(R))$ for some $s\geq 1$,}\\ &\text{then $k$ is a direct summand of $\operatorname{syz}_i^R(M)$ for some $i\leq s+4$ and for all $i\geq s+6$.}
    \end{align*}
\end{theorem}

Throughout this paper, we let $\I_1(A_j)$ be the $R$-ideal generated by the $1\times 1$ minors of the $j$-th matrix $A_j$ in some minimal free resolution. When the embedding dimension of $R$ is at least $2$, we have that $\I_1(A_j)=\mathfrak{m}$ for all $n\geq 8,s+7$, when the respective conditions of the above theorem are met, thus tying these results directly to ideal-periodicity. 


The Theorem above indicates that the resolution of interest is that of the Burch Ideal $\BI(R)$. In particular, if $k$ is direct summand of $\operatorname{syz}_s^R(\BI(R))$, then we get a similar result to the Burch Index $2$ and greater cases. However, Eisenbud and Dao show this is not always the case \cite{eisbur}: 
\begin{example}[Eisenbud and Dao (Ex 4.5)]
    Let $S=k[[a,b]]$, $I=(a,b^2)^2$. One can check $R=S/I$ has Burch index $1$. Let $M=R/(a,b^2)$. Then $\operatorname{syz}_1^R(M)=M^{\oplus 2}$, indicating no syzygy has $k$ as a direct summand. 
\end{example}

Thus $\I_1(A_j))=(a,b^2)$ for all $j$ and $A_j$ the matrices in some minimal free resolution of $M$. Thus 

\begin{prop}
    There exist rings $R$ with $\Burch(R)=1$ such that $\I_1(A_1)\neq \mathfrak{m}$ for some module $M/R$ and any $n$. 
\end{prop}

We begin weakening the restriction on the original Burch Ideal definitioin. However, we still restrict to the case where $I\neq 0$ for non-triviality, and remark that periodicity of $1\times 1$ minors is well understood in the regular local ring case. We initially care about cases where $\operatorname{depth}(S/I)=0$. The reason is twofold. First, from the Auslander-Buchsbaum formula, if the projective dimension of $M$ is finite, then 
$$\operatorname{pd}(M) + \dim R = \operatorname{depth}(M)$$ 
Thus if the projectve dimension of $M$ is finite, $M$ is free. Second, this condition, along with $I\neq 0$, ensures that $(I:\mathfrak{n})$ is a proper ideal of $R$, as it is well known a local Noetherian ring $R$ is depth $0$ iff $x\mathfrak{m}=0$ for some nonzero $x\in R$. In particular this allows us to form the bounds
$$\mathfrak{n}^2\subset \BI(I):=\mathfrak{n}I:(I:\mathfrak{n})\subset \mathfrak{n}$$
However, we will also consider positive depth rings $R$ in this paper, in which case $\BI(I)=R$ since $(I:\mathfrak{n})=I$. When definitions and theorems differ for positive depth rings, we will make a disclaimer. 

Let $I,N\subset S$ be ideals. We introduce $$\BI_N(I):=\mathfrak{n}I:(I:N)$$Note that unlike in the normal Burch ideal case, $\BI_N(I)$ is not necessarily contained in $N$, even in the case of depth $0$. This is because $(I:N)\supset (I:J)$ for all $J\supset N$, and this containment need not be strict. Thus let $J'=\cup J$ for all $J$ with $(I:N)=(I:J)$. We have that $$\mathfrak{n}J'\subset \BI_N(I)\subset J'$$ 


\begin{example}
    Let $S=k[[x,y]]$, $I=(x^2,xy,y^2)=\mathfrak{n}^2$, and $N=(y)$. Then $(I:(y))=(x,y)=(I:\mathfrak{n})$. In particular, the a priori bounds we have on $\BI_{(y)}(I)$ are 
    $$\mathfrak{n}^2\subset \BI_{(y)}(I)\subset \mathfrak{n}$$
    since, in the notation above, $J'=\mathfrak{n}$. Of course, here $\BI_{(y)}(I)=\mathfrak{n}^2\not\subset (y)$. 
\end{example}


Further, when $\operatorname{depth}(R)=0$, we let $$\Burch_{N}(I)= \operatorname{length}_S(N/(\BI_N(I)\cap N)) 
$$
$$\BI^0(I)=\mathfrak{n}$$
$$\BI^j(I)=\mathfrak{n}I:(I:\BI^{j-1}(I))=\BI_{\BI^{j-1}(I)}(I), \text{ } j>0$$
$$\Burch^j(I)= \operatorname{length}_S(\BI^{j-1}(I)/\BI^j(I)), \text{ } j>0 $$

\begin{prop}
    $\Burch^j(I)$ is well-defined. 
\end{prop}
\begin{proof}
    We must check that $\BI^j(I)\subset \BI^{j-1}(I)$. We use induction. When $j=1$, this is  true as $\BI^0(I)=\mathfrak{m}$ and $\BI^1(I)\subset \mathfrak{m}$ by the conditions on $S$. For the inductive step, we may assume $\BI^{j-1}(I)\subset \BI^{j-2}(I)$, in which case 
    $$\BI^j(I)=nI:(I:\BI^{j-1}(I))\subset nI:(I:\BI^{j-2}(I)) = \BI^{j-1}(I)$$
    since $(I:N)\supset (I:J)$ whenever $N\subset J$. 
\end{proof}

In the positive depth case, we keep the above definitions the same, except $\Burch^1(I):=\operatorname{length}_S(\mathfrak{n}/(\BI(I)\cap \mathfrak{n}))=0$. This then yields $\Burch^j(I)=0$, further noting that $\BI^j(I)=R$ for $j\geq 1$. As positive depth rings seem to quite 'un-Burch' rings, one may expect that periodicity is impossible to prove with Burch techniques. We show in Section \ref{untwist} that is not always the case. 

\begin{remark} 
    If $\Burch(I)\neq 0$, then $\operatorname{depth}(R)=0$. 
\end{remark}

\section{Burch Duality and Burch Closure}


\begin{defn}
    Consider an $m\times n$ matrix $A$. Throughout this paper let $[x]_p$ be the $m\times 1$ vector with $x$ as the $p$-th entry, and $0$ elsewhere. 
\end{defn}
    
\begin{remark}
    We often denote the reduction of $S$-ideals $N$ simply as $N$. Similarly we drop the reduction notation and interchangeably consider $x\in S$ an element of both $S$ and $R$. 
\end{remark}

\begin{defn}[Realization Set]
    The Realization Set of an ideal $N$, $\RS_I(N)$, is the set of elements $x^*\in (I:N)$ such that $x^*N\not\subset \mathfrak{n}I$. The Realized Set of an ideal $N$ is the difference $$\rs_I(N)=N-(\BI_N(I)\cap N)$$of sets.
\end{defn}

We also identify all elements in the realization and realized sets, respectively, that differ by multiplication by a nonzero element of $k$. 
Note that $\rs_I(N)$ is nonempty iff $\BI_N(I)\cap N\subsetneq N$ and also iff $\RS_I(N)$ is nonempty, both by definition. Thus the following remark: 
\begin{remark}
    $$\RS_I(N)\neq \emptyset  \Leftrightarrow \rs_I(N)\neq \emptyset  \Leftrightarrow \Burch_N(I)>0$$
\end{remark}

We say $x^*\in \RS_I(N)$ \emph{realizes} $x\in \rs_I(N)$ if $x^*x$ is a minimal generator of $I$. 

\begin{example}
    Let $S=k[[x,y,z]]$, $I=(x^2y,xy^2z,z^3)$, and $N=(x^2,y,z^2)$. Then $\BI_N(I)=(x,z^2,yz,y^2)$, $$\rs_I(N)=(x^2,y,z^2)-(x,z^2,yz,y^2)\cap (x^2,y,z^2)=(y)-(yz,y^2)$$ Since $(I:N)=(x^3,xyz,x^2y)$ and here $\RS_I(N)$ are precisely the elements of $(I:N)$ that \emph{realize} $y$,
    $$\RS_I(N)=\{xyz\}$$
\end{example}

\begin{lemma}[Burch Duality]
    $$\Burch_N(I)>0 \Rightarrow \Burch_{(I:N)}(I)>0$$
    Further, choose $x\in \rs_I(N)\neq \emptyset$. Then $x\in \RS_I((I:N))$. In particular, if $x^*\in \RS_I(N)$ realizes $x\in \rs_I(N)$, then $x\in \RS_I((I:N))$ realizes $x^*\in \rs_I((I:N))$. 
\end{lemma}
\begin{proof}
    Assume $\Burch_N(I)>0$. Let $x\in \rs_I(N)$. Then $x(I:N)\subset I$ and yet $x(I:N)\not\subset \mathfrak{n}I$. Thus $xx^*$ is a minimal generator of $I$ for some $x^*\in (I:N)$. In particular, $x^*\in \RS_I(N)$. 

    Now consider $\BI_{(I:N)}(I)=\mathfrak{n}I:(I:(I:N))$. Since $x(I:N)\in I$, $x\in (I:(I:N))$. But then since $xx^*$ is a minimal generator of $I$, $x\in \RS_I((I:N))$. Because $x^*(I:(I:N))\not\subset \mathfrak{n}I$ this also shows $x^*\not\in \BI_{(I:N)}(I)$, and since $x^*\in \RS_I(N)\subset (I:N)$, $x^*\in \rs_I((I:N))$. 
\end{proof}

\begin{remark}\label{book2}
    The Realized Set of $N$ is the set of elements $x\in N$ such that $x^*x$ is a minimal generator of $I$, for some $x^*\in \RS_I(N)\subset (I:N)$. Note that such an $x$ cannot be in $\BI_I(N)$, because $\BI_I(N)(I:N)\subset \mathfrak{n}I$.
\end{remark}



We use this duality to prove a certain general periodicity. To better understand the conditions of this lemma, consider Corollaries \ref{dual1} and \ref{dual2} immediately after. 
  
\begin{lemma}[Burch Dual $2$-Periods]\label{dual}
    Consider a minimal free resolution $\epsilon: (F_\bullet, A_\bullet)\rightarrow M$ over $R$. Let $\I_1(c_m)$ be the ideal generated by elements of a column $c_m$ of a minimal matrix representation of $A_m$. If ${J}\supset \I_1(c_m)$ for some $c_m$ that contains a reduction of some $x\in \rs_I(J)$, then for all $a\geq 1$,  
    \begin{align*}
        &{J}\subset \I_1(A_{m+2a})\\
        [{x}]_i &\text{ is a minimal generator of $\operatorname{im}(A_{m+2a})=\operatorname{syz}_{m+2a}(M)$}
    \end{align*}
for some $i$, and
    \begin{align*}
     &x^*R\subset \I_1(A_{m+(2a-1)})\\
         [{x^*}]_j &\text{ is a minimal generator of $\operatorname{im}(A_{m+(2a-1)})=\operatorname{syz}_{m+(2a-1)}(M)$}
    \end{align*}
    for each $x^*\in \RS_I(J)$ that realizes $x$, and some $j$. There is at least one such $x^*$. 
\end{lemma}

\begin{proof}
     Consider a morphism of free $S$-modules $B_m: F_{m}\rightarrow F_{m-1}$ that reduces to $A_m$ modulo $I$. Let $d_m$ be the corresponding lift of the column $c_m$. We can choose any $x^*\in \RS_I(J)$ that satisfies $x^*x\not\in I\mathfrak{n}$ for the $x\in d_m$ given in the theorem (cf. Remark \ref{book2}). Wlog let $d_m$ be the $j$-th columns of a minimal matrix representation of $B_m$. Thus the minimal generator of $F_m$, $e_j=[1]_j$, satisfies that $B_m(x^*e_j)\not\in \mathfrak{n}IG$. 

 $$x^*e_j\not\in B_m^{-1}(\mathfrak{n}IG)\supset \mathfrak{n}B_m^{-1}(IG)$$

 Reducing modulo $I$, $\overline{x^*e_j}\not \in \mathfrak{n}\ker{A_m}$. However, $x^*\in \RS_I(J)\subset (I:J)$, and thus $\overline{x^*e_j}\in (I:J)F$. Thus, $A_m([\overline{x^*}]_j)=x^*c_{m}=0$, and so $\overline{x^*e_j}\in \ker{A_m}$. Thus $\overline{x^*e_j}$ is a minimal generator of $\ker{A_m}$. 

 By exactness and invariance under quasi-isomorphism, $A_{m+1}$ can be written with $\overline{x^*e_j}$ as a column, say column $i$. Thus $[\overline{y}]_i\in \ker(A_{m+1})=\operatorname{im}(A_{m+2})$ for each $y\in J$.

 
 
Further since $x^*\in \RS_I(J)$ realizes $x$, by Burch Duality, $x\in \RS_I((I:J))$ realizes $x^*\in \rs_I((I:J))$ and $\Burch_{(I:J)}(I)\geq 1$. Thus we can carry out the above proof with $A_{m+1},$ choosing $J=\I_1(x^*e_j)\subset S$, and swapping $x,x^*$ to reach the conclusion.


\end{proof}

A corollary of this theorem looks more familiar, and generalizes the case of the standard Burch Index $\BI(I)$ in Proposition $4.3$ of Eisenbud and Dao \cite{eisbur}:

\begin{corollary}\label{dual1}
    If $\I_1(A_m)\not\subset \BI_N(I)$ and $N\supset \I_1(A_m)$, then $N\subset \I_1(A_{m+2a})$ for $a\geq 1$. Thus if $\Burch_N(I)\geq 1$, and $\I_1(A_m)=N$, $N\supset \I_1(A_{m+2a})$.  
\end{corollary}

In fact we have shown the more general condition: 

\begin{corollary}\label{dual2}
    If $\I_1(c_m)\not\subset \BI_N(I)$ and $N\supset \I_1(c_m)$, then $N\subset \I_1(A_{m+2a})$ for $a\geq 1$. Thus if $\Burch_N(I)\geq 1$, and $\I_1(c_m)=N$, $N\supset \I_1(A_{m+2a})$.  
\end{corollary}

\begin{example}
    Let $S=k[[x,y,z]]$, $I=(x^2y,y^2z,z^2x)$, and $N=(x^2,y^2,z^2)$. Then $\Burch(I)=0$, $\Burch_N(I)=0$, and yet if we resolve $N$ via $(F_\bullet,A_\bullet)$ we see
    \begin{align*}
        \I_1(A_0)&=(x^2,y^2,z^2)\\
        \I_1(A_j)&=(x,y,z)
    \end{align*}
    for $j\geq 1$, because $\Burch_{(x)}(I),\Burch_{(x^2)}(I)>0$, and similarly for $y,y^2,z,z^2$. 
\end{example}

\begin{remark}
    By convention we consider $A_0$, the matrix whose columns are the minimal generators of $N$, when resolving an ideal $N$. When resolving a module $M$, the matrix whose cokernel is $M$ is $A_1$. 
\end{remark}
In addition, we have shown the following useful fact about positive depth ring resolutions: 

\begin{corollary}\label{dualpos}
    If the conditions of the Lemma are satisfied for any $A_m$ in a resolution of $M$, then $M$ has infinite projective depth. In particular this holds if $\Burch_{\I_1(A_j)}(I)>0$. This is true in the non-trivial setting where $\operatorname{depth}(R)>0$.  
\end{corollary}
\begin{proof}
    The lemma proves the existence of $[x^*]_j$ in $A_{m+(2a-1)}$ for $a\geq 1$. Thus $A_\bullet$ is not $0$ asymptotically. 
\end{proof}

\section{Proof of Generalized Burch Index Theorem}\label{kemma}
We are now equipped to prove the result on generalized Burch Index. First we need a helper lemma. 
\begin{lemma}[Minimal Generators of Tor]\label{tor}
Consider $\operatorname{Tor}(X,Y)$ of $R$-modules $X,Y$ with minimal free resolutions $\epsilon: (F_\bullet,A_\bullet)\rightarrow X$ and $\delta: (G_\bullet,B_\bullet)\rightarrow Y$. If there exists a nonzero minimal generator $e$ of $F_j$ and $f$ of $G_j$. If $e\otimes f$ is in $\ker(A_j\otimes Y)$, then $e\otimes f$ is nonzero and a minimal generator of $\operatorname{Tor}_j(X,Y)$. 

\begin{proof}
Because $F_\bullet$ is a free $R$-module and $\epsilon$ is a minimal free resolution, we have $\I_1(A_{j+1})\subset \mathfrak{m}$ and thus $e\not\in \operatorname{im}(A_{j+1})$. Assume that $e\otimes f\in \operatorname{im}(A_{j+1}\otimes Y)$, then $e\otimes f=e'\otimes f'$ for some $e'\in \operatorname{im}(A_{j+1})$ and $f'\in Y$. But this cannot be the case because $f$ is a minimal generator of $Y$, so we cannot choose $e',f'$ such that $e'$ has $R$-coordinates in $\mathfrak{m}$. Thus $e\otimes f$ is nonzero. 

To show $e\otimes f$ is a minimal generator, assume for contradiction that, with $x_m\in \mathfrak{m}$,
\begin{align*}
    e \otimes f &= \sum_m x_m e_m\otimes f_m\mod \operatorname{im}(A_{j+1}\otimes Y)\\
    \Leftrightarrow e \otimes f &= (\sum_m x_m e_m\otimes f_m) + (E\otimes F)
\end{align*}
where $E\otimes F\in \operatorname{im}(A_{j+1}\otimes Y)$. We can assume that we cannot factor out $y| 1\neq y\in R$ from $e_m,f_m$. For this sum to be equal to $e\otimes f$, we need the summands on the RHS to reduce to 
$$\sum_m e\otimes x_mf_m + e\otimes F'$$
or
$$\sum_m x_me_m\otimes f + E'\otimes f$$
Since the first paragraph of the proof shows that $E\otimes F$ is not in $\operatorname{im}(A_{j+1}\otimes Y)$ if $E$ and $F$ are minimal generators, $E'$ and $F'$ must not be minimal generators of their respective modules. But then the wlog the first case reduces to $e\otimes (\sum_m x_mf_m+F')$, which is not equal to $e\otimes f$ since $f$ is a minimal generator. 
\end{proof}

\end{lemma}

This allows us to prove the following major lemma. After reading through the lemma, consider again this remark: 
\begin{remark}[Intuition]\label{intuition}
    If $\Burch_N(I)\geq 1$, then ideals $N$ tend to persist in resolutions (cf. Corollary \ref{dual1}). If $\Burch_N(I)\geq 2$, then ideals $N$ not only persist, but also any ideals less than $N$ tend to grow to $N$.  
\end{remark}

\begin{lemma}\label{Up}
    Let $\epsilon: (F_\bullet,A_\bullet)\rightarrow M$ be a minimal $R$-resolution, and $M$ be an $R$-module. Let $\Burch_N(I)\geq 2$, and $(0)\subsetneq \I_1(A_{j})\subset N$ for some $1\leq j$. Then for all $v\geq j+5$,
    \begin{equation}\label{3}
        \I_1(A_{v})+\I_1(A_{v+1})\supset N
    \end{equation} 
\end{lemma}
\begin{proof}
    We will first show \begin{equation}\label{eq}
        \I_1(A_m)+\I_1(A_{m+1})\neq X\subsetneq N
    \end{equation}
    for $m=j+2$. For contradiction assume $j=1$ and that $\I_1(A_m)+\I_1(A_{m+1})=X\subsetneq N$ for $m$ odd. Then there is an ideal $Q$ with $X\subset Q\subsetneq N$ and $\operatorname{length}_S(N/Q)=1$. Tensoring the resolution $(F_\bullet,A_\bullet)$ with $R/Q$, one has $A_{m}\otimes R/Q=A_{m+1}\otimes R/Q=0$ and thus gets
    \begin{equation}\label{1}
        \oplus R/Q \simeq \operatorname{Tor}_m(M,R/Q)\simeq \operatorname{Tor}_{m-1}(M,Q)\simeq \operatorname{Tor}_{m-2}(\operatorname{im}(A_1),Q)
    \end{equation}
Consider a minimal resolution $\delta: (G_\bullet,B_\bullet)\rightarrow Q$. Resolving $Q\otimes \operatorname{im}(A_1)$ over $R\otimes \operatorname{im}(A_1)$, one considers
$$G_{m-1}\otimes \operatorname{im}(A_1) \xrightarrow{B_{m-1}\otimes \operatorname{im}(A_1)} G_{m-2}\otimes \operatorname{im}(A_1)\xrightarrow{B_{m-2}\otimes \operatorname{im}(A_1)} G_{m-3}\otimes \operatorname{im}(A_1)$$
which we rewrite as 
$$H_{m-1}\xrightarrow{C_{m-1}} H_{m-2}\xrightarrow{C_{m-2}} H_{m-3}$$ for brevity. Since $\Burch_N(I)\geq 2$, and $N\supset Q\supset \I_1(B_0)$, and $Q$ has colength $1$ in $N$, we may apply Lemma \ref{dual} to show the existence of a minimal generator $[x^*]_p$ of $\operatorname{im}(B_{1+2b})$ with $x^*\in \RS_I(N)$ and $b\geq 0$, for some $p$. Thus, $[n]_q\in \operatorname{im}(B_{2+2b})$ for each $n\in N$ and some $q$. 

Note that $m-2= 1$ is odd by assumption and thus $[x^*]_p$ is a minimal generator of $\operatorname{im}(B_{m-2})$. Then $q$ satisfies $e=[1]_q\otimes g\in H_{m-2}$ and $$C_{m-2}(e)=[x^*]_p\otimes g\simeq [1]_p\otimes x^*g=0$$ since $x^*\in \RS_I(N)\subset (I:N)$ and $g$ is a column of $A_1$ which has $\I_1(A_1)\subset N$ by assumption. Thus $e\in \ker C_{m-2}$. 

Note that $[1]_q$ is nonzero and $g$ is nonzero by assumption since $(0)\subsetneq \I_1(A_j)\subset N$. Thus by Lemma \ref{tor}, $e$ is nonzero and a minimal generator of $\operatorname{Tor}_{m-2}(\operatorname{im}(A_1),Q)$. Yet for all $n\in N$, $$ne=[n]_q\otimes g\in \operatorname{im} C_{m-1}$$ and thus $ne=0$ in $\operatorname{Tor}_{m-2}(\operatorname{im}(A_1),Q)$. Choose $n\not\in Q$ to get a contradiction with Equation (\ref{1}): $\oplus R/Q$ has no minimal generator with order $n$. To ensure $m$ is odd, we include $A_m,A_{m+1},A_{m+2}$ in (\ref{eq}). We can choose arbitrary $j\geq 1$ by considering the resolution of $\operatorname{syz}_{j-1}(M)$. To show (\ref{3}), note that (\ref{eq}) implies $\I_1(A_{y})\not\in \BI_N(I)$ for some $y\in \{j,j+1,j+2\}$, and use Lemma \ref{dual}. 

\end{proof}

\begin{proof}[Proof of Theorem \ref{Big1}]
    We can exclude free modules $M/R$, which have $0$ resolution. All other cases with finite projective dimension are not included in Theorem \ref{Big1} since $\Burch(I)\neq 0$ excludes positive depth rings. Assume $\operatorname{bd}(I)=j$. Then $\Burch^{j+1}(I)\geq 2$. Thus let $\epsilon: (F_\bullet, A_\bullet)$ be a minimal free resolution with $\I_1(A_1)\subset \BI^{j}(I)$. By Lemma \ref{Up}, $$\I_1(A_v)+\I_1(A_{v+1})\supset \BI^{j}(I)$$ for $v\geq 6$. Alternatively, if $$\I_1(A_1)\not\subset \BI^{j}(I)$$, then $$\I_1(A_{1+2a})\supset \BI^j(I)$$ for $a\geq 1$ by Corollary \ref{dual1}.
    Thus all resolutions satisfy $$\I_1(A_v)+\I_1(A_{v+1})\supset \BI^{j}(I)$$ for $v\geq 6$. And for all such resolutions, if $$\I_1(A_v)+\I_1(A_{v+1})\supsetneq \BI^{j}(I)$$, then wlog $\I_1(A_{v})\not\subset \BI^j(I)$, and $$\I_1(A_{v+2a})\supset \BI^{j-1}(I)$$ by Corollary \ref{dual1}. By induction, one sees that all resolutions satisfy $$\I_1(A_v)+\I_1(A_{v+1})=\BI^q(I)$$ for some $0\leq q\leq j$. 
\end{proof}

Finding a non-trivial example where each $\BI^n(I)$ does not quickly degenerate to the maximal ideal is not as easy as one might expect. Yet we would like one to verify the validity of these lemmata. Here is one: 
\begin{example}
    Let $S=k[[x_1..x_3]]$ and $I=(x_2x_3+28x_3^2,x_2^2-30x_3^2,x_1x_3^2,x_1^3x_3)$. Then $\operatorname{gb}(I)=2$ and $\operatorname{bd}(I)=1$. Importantly, the minimal $R$-resolution $\epsilon: (F_\bullet, A_\bullet)\rightarrow \BI^1(I)$ has $\I_1(A_6)=\I_1(A_7)=\BI^1(I)=(x_3,x_2,x_1^2)$. One notes that 
    $$\BI^2(I)=(x_2+28x_3,x_3^2,x_1x_3,x_1^3)$$
    and considers the minimal resolution of $(x_2+28x_3)\subsetneq \BI^2(I)$: 
\begin{align*}
    \I_1(B_1)&=(x_2+28x_3)\\
    \I_1(B_2)&=(x_3,x_1x_2)\\
    \I_1(B_3)&=(x_3,x_2,x_1^3)\\
    \I_1(B_j)&=(x_3,x_2,x_1^2), \text{ } 4\leq j \leq 8
\end{align*}
Lemma \ref{Up} says that $\I_1(B_3)+\I_1(B_4)$ is not strictly contained in $\BI^1(I)$, and together with Lemma \ref{dual} further implies that $\I_1(B_j)$ or $\I_1(B_{j+1})$ contains $\BI^1(I)$ for $j\geq 4$. These conclusions are supported with this example. 
\end{example}


\begin{remark}
    One can analogously define $\BI_N^j(I)$, $\Burch_N^j(I)$, $\operatorname{bd}_N(I)$, and $\operatorname{gb}_N(I)$, and show the same type of result for $\operatorname{gb}_N(I)\geq 2$. Note this will have the same periodic result unless $\I_1(A_m)\supsetneq N$, after which the resolutions are not proven to be periodic. Note we must define as an edge case $\Burch_N^j(I)=0$ when $\BI_N(I)\supsetneq N$ for the same reason that comes up for positive depth case in the $N$-Burch Ideal section. 
\end{remark}

\section{Untwisting}\label{untwist}

We now seek to generalize the results that relied on Lemma \ref{Up} and Corollary \ref{dual1} to results relying on Corollary \ref{dual2}. Here is the thesis of this section: Because of the column-wise periodicity we see in Lemma \ref{dual} and Corollary \ref{dual1}, we can more easily prove periodicity when the elements of each column of $A_m$ generates a small ideal. For example, when resolving an ideal $N$ of $R$, $\I_1(A_0)=(n_1,...,n_j)$ where $(n_1,...,n_j)$ is some minimal generating set of $N$. Take $N=\mathfrak{m}=(x_1,...,x_n)$ and assume that $\Burch(I)=0$. Then without the column-wise approach, we cannot say anything about the periodicity of this sequence in general without the Burch approach. However, if $\Burch_{(x_j)}(I)\geq 1$ for each $x_j$, then $\I_1(A_{2m})=\mathfrak{m}$ for each $m$ by Lemma $\ref{dual}$, despite $I$ not having a positive Burch Index (in fact for the maximal ideal case we may say more and still get a direct summand of $k$ in the kernel as in Eisenbud and Dao \cite{eisbur}, but this is beside the point). A similar fact holds true for $\Burch_N(I)=0$.

\begin{example}[Positive Depth]
    One straightforward way to get $\Burch(I)=0$ yet $\Burch_{(x_j)}(I)\geq 1$ for each $j$ is to choose $I$ as a monomial ideal generated by a single element. This can yield \textit{positive depth} cases as well. For example, let $S=k[[x,y]]$ and $I=(x^2y)$. This meets the conditions and we see the resolution begins
    $$R^2 \xrightarrow{\begin{pmatrix}
        x^2 & 0\\x & y
    \end{pmatrix}} R^2 \xrightarrow{\begin{pmatrix}
        -x & x^2\\y & 0
    \end{pmatrix}} R^2 \xrightarrow{\begin{pmatrix}
        xy & 0\\y & x
    \end{pmatrix}}R^2 \xrightarrow{\begin{pmatrix}
        -y & xy\\x & 0
    \end{pmatrix}} R^2 \xrightarrow{\begin{pmatrix}
        x & y
    \end{pmatrix}} R$$
    and whose matrices repeat with period $2$ (there are two pairs of isomorphic images in the presumed $4$-period above), with $${\begin{pmatrix}
        -y & xy\\x & 0
    \end{pmatrix}}$$ being the next matrix in the resolution. The resolution of the maximal ideal may be able to be obtained through other means, but even in the case where $\I_1(A_1)=\mathfrak{m}$, these techniques are more general: in the same ring, consider the resolution of
    $$\begin{pmatrix}
        x & xy \\ y^2 & y
    \end{pmatrix}$$
    which by column-wise Burch will have $\I_1(A_j)=\mathfrak{m}$ for each odd index $j$. Again this particular case turns out to be a complete intersection ring, where ideal-periodicity was proven by Dao, Brown, and Sridhar \cite{dao2}. More general examples are of course available when the ring is not a complete intersection ring (or Golod) and for general ideals $N\neq \mathfrak{n}$, as in the next theorem. 
\end{example}


\begin{theorem}
    Let $I$ be a monomial ideal such that $\Burch(I)=1$ and $S/I$ is a depth $0$ ring. If $x_n^2\alpha$ is a minimal generator of $I$ for some $\alpha\in S$ and $\epsilon: (F,A)\rightarrow \BI(I)$ is a minimal free resolution, then $$\I_1(A_m)+\I_1(A_{m+1})\in \{\BI(I),\mathfrak{m}\}$$ for all $m$. If $\I_1(A_{m_0})+\I_1(A_{m_0+1})=\mathfrak{m}$ for some $m_0$, then $\I_1(A_m)=\mathfrak{m}$ for all $m>>m_0$. In particular, $F$ is ideal-periodic with period at most $2$. 
\end{theorem}
\begin{proof}
    Since $S/I$ has depth $0$ and $I$ is monomial, $x_j\in \rs_I(x_j)$ for each $j$ with $1\leq j<n$, and since $x_n^2\alpha$ is a minimal generator, $x_n^2\in \rs_I(x_n^2)$. In particular, $\Burch_{(x_j)}(I)>0$ for each $j$ with $1\leq j<n$, and $\Burch_{(x_n^2)}(I)>0$. By Lemma \ref{dual} and specifically Corollary \ref{dual2}, since $\BI(I)=(x_1,...,x_n^2)=\I_1(A_1)$, $\BI(I)\subset I_1(A_{2a+1})$ for all $a\geq 0$. For the second claim, see Eisenbud and Dao Theorem $4.1$ \cite{eisbur}. Note we can also apply Lemma \ref{dual} again for the second claim, but due to its generality would only get $2$-periodicity of $\mathfrak{m}$. 
\end{proof}

This occurs often and is not an accident; if $x\in \rs_I(N)$, then $x\in \rs_I(J)$ for any $J\subset N$, as long as $x\in J$. Thus $N$-periodicity that can be obtained via Lemma \ref{dual} by considering the entire ideal $\I_1(A_1)$ can always be obtained by considering just the column containing a $x\in \rs_I(N)$ (also by Lemma \ref{dual}). When choosing $N=\I_1(c)$ for some column however, the converse is not true. 

If instead we tried to resolve $$M=\operatorname{coker}
    \begin{pmatrix}x_1&x_2x_1&...&x_2\\ x_2&x_{n}^2&...&x_{1}^2x_3\\ ...&...&...&...\\ x_n&x_{n-1}&...&x_{n}^3\end{pmatrix}
$$
we could not use the column-wise approach with $\Burch_{(x_j)}(I)\geq 1$, since the column ideals are too large. We loosely call this phenomenon a \textit{twisted} matrix, and the results of this section are devoted to recovering the column-wise approach for such matrices, or \textit{untwisting}. Both of the techniques in this section will be based on Lemma \ref{Up}. 



\begin{lemma}[Untwisting Supplement, Compare w/ Lemma \ref{Up}]\label{twist1}
     Let $\epsilon: (F_\bullet,A_\bullet)\rightarrow M$ be a minimal $R$-resolution and consider an ideal $N$. Let $\Burch_{(n)}(I)\geq 1$ for each $n$ in some minimal generating set of $N$, and $(0)\subsetneq \I_1(c_{j})\subset (n)$ for some column $c_j$ of $A_j$ with $1\leq j$. Further consider each ideal $Q\subset N$ with $\operatorname{length}_S(N/Q)=1$ and the minimal resolution $\delta: (G_\bullet, B_\bullet)\rightarrow Q$. Let $n\not\in Q$ be the unique minimal generator of $N$ not in $Q$. Assume $d_i=[n^*]_q$, where $d_i$ is a column in $B_i$ for some $i\geq j$ and $n^*\in \RS_I((n))$. Then 
    \begin{equation}\label{32}
        \I_1(A_{v})+\I_1(A_{v+1})\supset N
    \end{equation} for all $v\geq j+5$. 
\end{lemma}

In other words, under certain conditions, we can check whether the module $M$ has $N$-periodicity by considering column-wise Burch periodicity of the \emph{ideals} $Q$, which have unmixed columns as the minimal generators of $Q$.

\begin{proof}[Proof of Lemma \ref{twist1}]

    We follow a similar proof as Lemma \ref{Up}. We first show
    \begin{equation}\label{eq2}
        \I_1(A_m)+\I_1(A_{m+1})\neq X\subsetneq N
    \end{equation}
    for $m=i+2$. Thus assume for contradiction equality with some $X\subsetneq N$ and let $Q$ be an ideal such that $\I_1(A_m)+\I_1(A_{m+1})+\I_1(A_{m+2})=X\subset Q\subsetneq N$ and $Q$ has colength $1$ in $N$. Let $n$ be a minimal generator of $N$ such that $n\not\in Q$. Let $j$ be the index such that $\I_1(c_j)\supset (n)$ and $1\leq j=m-2$. We may as before assume $j=1$, lest we consider the resolution of $\operatorname{syz}_{j-1} (M)$. As before consider 
    $$\oplus R/Q\simeq \operatorname{Tor}_m(M,R/Q)\simeq \operatorname{Tor}_{m-2}(\operatorname{im}(A_1),Q)$$
    and let 
    $$H_{m-1}\xrightarrow{C_{m-1}} H_{m-2}\xrightarrow{C_{m-2}} H_{m-3}$$
    be defined as in Lemma \ref{Up}. By Lemma \ref{dual} and the conditions on the resolution of $Q$, have that $[n^*]_q$ is a column in $B_{2b+i}$, and $[x]_p\in \operatorname{im} B_{2b+1+i}$ for each $x\in N$, $b\geq 0$. 


    Since by assumption $m-2-i$ is even, we have $[n^*]_q$ as a column of $B_{m-2}$. We can choose the minimal generator $g$ of $\operatorname{im}(A_1)$ corresponding to $c_j$. Then $q$ satisfies $e=[1]_p\otimes g\in H_{m-2}$ and 
    $$C_{m-2}(e)=[n^*]_q\otimes g\simeq [1]_p\otimes n^*g=0$$
    since $g$ is a column of $A_1$ whose entries generate $\I_1(c_j)\subset (n)$. By Lemma \ref{tor}, $e$ is nonzero and a generator of $\operatorname{Tor}_{m-2}(\operatorname{im}(A_1),Q)$. Yet $$ne=[n]_p\otimes g\in \operatorname{im} C_{m-1}$$ since $[n]_q\in \operatorname{im} B_{2b+1+i}$. The conclusion now follows the same as in Lemma \ref{Up}.

\end{proof}


Note these conditions yield more results than just considering when $\Burch_N(I)>0$ for the entire ideal $N=\I_1(A_1)$. One class of examples of ideals that satisfy the conditions if the Lemma yield the second main theorem of this paper. 

\begin{proof}[Proof of Theorem \ref{Big2}]
It suffices to show the conditions of the Theorem satisfy those of Lemma \ref{twist1} when $N=\mathfrak{n}$. Thus it suffices to show: If for each $i$ there is some $j\neq i$ such that there exists $\alpha\in \RS_I((x_i))\cap \RS_I((x_j))$, then $\Burch_{(x_i)}(I)=\Burch_{(x_j)}(I)\geq 1$ and furthermore we may apply the above lemma when $N=\mathfrak{n}$. 

Since each colength $1$ ideal $Q$ contains $x_i$ or $x_j$, and the conditions imply that Lemma \ref{dual} implies that $[\alpha]_p$ will be a column in the second matrix in each resolution. Since $\alpha$ realizes both $x_i$ and $x_j$, $[x_i]_q$ and $[x_j]_q$ will be a column in the third matrix in the resolution. Since the $N$-Burch indices are positive, we are done. 
\end{proof}

 \section{Future Work}
The specific conditions that allow application of the lemma in the Untwisting section should be further fleshed out. More heuristics on the proportion of Burch Index $0$ rings that can be untwisted with the lemma should be explored. Since the lemma in the Untwisting section has a condition where ideals generated by entries of in some columns of $A_j$ are contained in an ideal $N$ for some $j$, one should explore 'resolving' a resolution backwards. In particular, if $\I_1(A_j)\supsetneq N$ for all $j<m$ in the resolution of a module $M$, for example, can we resolve a module $J$ such that the matrices $B_i$ in its resolution satisfy $\I_1(B_{i+b})=\I_1(A_{i})$ for some $b>0$ and $I_1(B_{i})\subsetneq N$ for some $i<b$? The Gorenstein case seems like a good place to start. 
 
There is also a notion of Burch Closure, where the successive realization sets, $\RS_I(\langle \RS_I (... \langle \RS_I(N) \rangle ...) \rangle$, and corresponding realized sets appear in a resolution with, for example, $\I_1(A_1)=N$ and $\Burch_N(I)\geq 1$. Heuristics and results on how these blow up should be considered. 
 
\printbibliography

\end{document}